\documentclass{article}
\usepackage{amsfonts}
\usepackage{amsmath}

\newtheorem{theorem}{Theorem}
\newtheorem{acknowledgement}[theorem]{Acknowledgement}

\newtheorem{definition}[theorem]{Definition}

\newtheorem{lemma}[theorem]{Lemma}

\newtheorem{proposition}[theorem]{Proposition}
\newtheorem{remark}[theorem]{Remark}

\newenvironment{proof}[1][Proof]{\noindent\textbf{#1.} }{\ \rule{0.5em}{0.5em}}
\input{tcilatex}

\begin{document}

\title{The Cauchy interlace theorem for symmetrizable matrices}
\author{Said Kouachi \\
Department of Mathematics, College of Sciences,\\
Qassim university, KSA}
\maketitle

\begin{abstract}
Symmetrizable matrices are those which are symmetric when multiplied by a
diagonal matrix with positive entries. The Cauchy interlace theorem states
that the eigenvalues of a real symmetric matrix interlace with those of any
principal submatrix (obtained by deleting a row-column pair of the original
matrix). In this paper we extend the Cauchy interlace theorem for symmetric
matrices to this large class, called symmetrizable matrices. This extension
is interesting by the fact that in the symmetric case, the Cauchy interlace
theorem together with the Courant-Fischer minimax theorem and Sylvester's
law of inertia, each one can be proven from the others and thus they are
essentially equivalent. The first two theorems have important applications
in the singular value and eigenvalue decompositions, the third is useful in
the development and analysis of algorithms for the symmetric eigenvalue
problem. Consequently various and several applications whom are contingent
on the symmetric condition may occur for this large class of not necessary
symmetric matrices and open the door for many applications in future
studies. We note that our techniques are based on the celebrated Dodgson's
identity \cite{Dod}.\newline
\newline
\textbf{Math. Subj. Classification 2010: }15A15, 15A18, 15B57, 11C20.
\end{abstract}

\section{Introduction.}

The Cauchy interlace theorem states that the eigenvalues of a real symmetric
matrix of order m interlace with those of any principal submatrix of order m
- 1. The idea behind to extend this theorem to symmetrizable matrices came
when we calculated, explicitly, the eigenvalues of a class of m order
Tridiagonal, Pentadiagonal and Heptadiagonal matrices (see \cite{Kou ELA}, 
\cite{Kou LJM} and \cite{Kou IJPAM}) whom are not necessary symmetric but
symmetrizable. We observed that their eigenvalues interlace with those of
their corresponding m-1 order principal submatrices.\newline
We begin by presenting a short and simple proof of the theorem based on the
well known Dodgson's Algorithm, then we prove its extension. At our
knowledge, this theorem is applicable only to symmetric matrices and our aim
is not to say that our result is new but we hope to provide the reader with
another technique to prove the Cauchy interlace theorem.

Proofs of this well known theorem have been based on Sylvester's law of
inertia \cite{Par}, the Courant-Fischer minimax theorem (see \cite{Gol-Van}, 
\cite{Hor-Joh}, \cite{Ike-Ina-Miy} and \cite{Bel}) and others more simple
are based on some properties of polynomials \cite{Fis}.

\section{The Cauchy interlace theorem for real symmetric matrices.}

The following theorem presents the more simplified form of the Cauchy
interlace theorem \cite{Cau}.

\begin{theorem}
(\textbf{The Cauchy interlace theorem})\label{The Cauchy interlace theorem}
If a row-column pair is deleted from a real symmetric matrix, then the
eigenvalues of the resulting matrix interlace those of the original one.
\end{theorem}

\begin{proof}
The celebrated Dodgson's identity \cite{Dod} states, for a square matrix $%
A=\left( a_{i,j}\right) _{1\leq i,j\leq m}$, the following 
\begin{equation}
\left. 
\begin{array}{c}
\det \left[ \left( a_{i,j}\right) _{1\leq i,j\leq m}\right] \det \left[
\left( a_{i,j}\right) _{\substack{ i\neq k,l \\ j\neq k,l}}\right] = \\ 
\\ 
\det \left[ \left( a_{i,j}\right) _{\substack{ i\neq l \\ j\neq l}}\right]
\det \left[ \left( a_{i,j}\right) _{\substack{ i\neq k \\ j\neq k}}\right]
-\det \left[ \left( a_{i,j}\right) _{\substack{ i\neq l \\ j\neq k}}\right]
\det \left[ \left( a_{i,j}\right) _{\substack{ i\neq k \\ j\neq l}}\right] ,%
\end{array}%
\right.   \label{dodgson}
\end{equation}%
\newline
for all $m>2$. If $\lambda _{1}<\lambda _{2}<....<\lambda _{m}$ lists the
eigenvalues of $A$, then by application of the above identity to the matrix $%
A-\lambda _{p}I$ for some fixed $1\leq p\leq m$, we deduce%
\begin{equation}
\left. 
\begin{array}{c}
\det \left[ \left( a_{i,j}\right) _{\substack{ i\neq l \\ j\neq l}}-\lambda
_{p}I_{m-1}\right] \det \left[ \left( a_{i,j}\right) _{\substack{ i\neq k \\ %
j\neq k}}-\lambda _{p}I_{m-1}\right] = \\ 
\\ 
\det \left[ \left( a_{i,j}\right) _{\substack{ i\neq l \\ j\neq k}}-\lambda
_{p}I_{m-1}\right] \det \left[ \left( a_{i,j}\right) _{\substack{ i\neq k \\ %
j\neq l}}-\lambda _{p}I_{m-1}\right] .%
\end{array}%
\right.   \label{identity}
\end{equation}%
Since the matrix $A-\lambda _{p}I$ \ is symmetric, then%
\begin{equation}
\det \left[ \left( a_{i,j}\right) _{\substack{ i\neq l \\ j\neq k}}-\lambda
_{p}I_{m-1}\right] =\det \left[ \left( a_{i,j}\right) _{\substack{ i\neq k
\\ j\neq l}}-\lambda _{p}I_{m-1}\right] ,  \label{formula}
\end{equation}%
and this shows that all the principal minors of order $m-1$ of the matrix $%
A-\lambda _{p}I$ $\ $have the same sign for a fixed $p=1,...,m.$ By using a
well known properties of the characteristic polynomial of the matrix $%
A-\lambda _{p}I$ $\ $which has $p-1$ negative eigenvalues, one null and $m-p$
positive, we deduce that the product $\Pi _{p-1}$ of the non zero
eigenvalues is equal to the sum of all its principal minors of order $m-1$
whom have the same sign as $\left( -1\right) ^{p-1}.$ Then, for a fixed $%
1\leq k\leq m$, the characteristic polynomial of the matrix $\left[ \left(
a_{i,j}\right) _{\substack{ i\neq k \\ j\neq k}}\right] $ 
\begin{equation}
P_{k}\left( \lambda \right) =\det \left[ \left( a_{i,j}\right) _{\substack{ %
i\neq k \\ j\neq k}}-\lambda I_{m-1}\right] ,  \label{Pk}
\end{equation}%
which is of degree $m-1,$ satisfies the following inequalities%
\begin{equation}
\left( -1\right) ^{p-1}P_{k}\left( \lambda _{p}\right) >0,\ \ \ \ p=1,...,m.
\label{Alter}
\end{equation}%
This ends the proof of the Cauchy interlace theorem for a real symmetric
matrix.
\end{proof}

\section{Results and Proofs.}

To introduce the class of matrices (not necessary symmetric) for whom the
Cauchy interlace theorem holds, we begin with the following

\begin{definition}
A square matrix $A$ of order $m>2$ is called symmetrizable if it is sign
symmetric, i.e.%
\begin{equation}
a_{ij}=a_{ji}=0\text{ or }a_{ij}.a_{ji}>0,\ i\neq j=1,2,...,m.  \label{sign}
\end{equation}%
and if for all permutation $\sigma $ of the set of integers $\left\{
1,2,...,m\right\} $, we have%
\begin{equation}
\underset{i=1}{\overset{m}{\Pi }}a_{i.\sigma _{i}}=\underset{i=1}{\overset{m}%
{\Pi }}a_{\sigma _{i}.i}.  \label{eq}
\end{equation}
\end{definition}

When $m=2,$ we simply suppose $a_{12}.a_{21}>0$ and when $m=3$, by
application of the definition, $A$ is symmetrizable if (\ref{sign}) is
satisfied and $a_{12}.a_{23}.a_{31}=a_{21}.a_{32}a_{13}$; since the other
equalities (\ref{eq}) are automatically satisfied. For example, all the
matrices%
\begin{equation*}
A=\left( 
\begin{array}{ccc}
a & -3 & 1 \\ 
-6 & b & 4 \\ 
5 & 10 & c%
\end{array}%
\right) ,
\end{equation*}%
are symmetrizable for all real numbers $a,$ $b$ and $c.$\newline
In \cite{Fom}, Lemma 3.2, the authors characterized symmetrizable matrices
as follows

\begin{proposition}
A square matrix of order m is symmetrizable if and only if it is symmetric
by sign and for all k=3,...,m, we have%
\begin{equation*}
a_{i_{1},i_{2}}a_{i_{2},i_{3}}...a_{i_{k},i_{1}}=a_{i_{2},i_{1}}a_{i_{3},i_{2}}...a_{i_{1},i_{k}},
\end{equation*}%
for all finite sequence $i_{1},\ i_{2},...,i_{k}$.\newline
However, we present a more simplified version of the above proposition and
give the following characterization of symmetrizable matrices.
\end{proposition}

\begin{proposition}
A square matrix of order m is symmetrizable if and only if all its principal
submatrices of order m-1 are symmetrizable.
\end{proposition}

\begin{proof}
First the matrix is symmetric by sign if and only if all its principal
submatrices are too. If A is symmetrizable, then by choosing in (\ref{eq})
any permutation $\sigma $ satisfying $\sigma _{i}=i$, for a fixed $%
i=1,2,...,m$, we can conclude that any principal submatrix of a
symmetrizable matrix A is also. Suppose that all principal submatrices of
order m-1 are symmetrizable. By writing $i_{k+1}=\sigma \left( i_{k}\right)
,\ k=1,...,m-1$ with $i_{1}=1$, then $i_{m+1}=\sigma _{1}^{m}=1$ and for all 
$j=2,...,m-2$, we have%
\begin{equation*}
\underset{i=1}{\overset{m}{\Pi }}a_{i.\sigma _{i}}=:\underset{k=1}{\overset{m%
}{\Pi }}a_{i_{k},i_{k+1}}=\left( \underset{k=1}{\overset{j}{\Pi }}%
a_{i_{k},i_{k+1}}\right) \left( \underset{k=j+1}{\overset{m}{\Pi }}%
a_{i_{k},i_{k+1}}\right) .
\end{equation*}%
Since the (j+1) order principal submatrix with row-column pairs $%
i_{1},i_{2},...,i_{j+1}$\ is symmetrizable, then we have%
\begin{equation*}
\left( \underset{k=1}{\overset{j}{\Pi }}a_{i_{k},i_{k+1}}\right)
.a_{i_{j+1},i_{1}}=\left( \underset{k=1}{\overset{j}{\Pi }}%
a_{i_{k+1},i_{k}}\right) .a_{i_{1},i_{j+1}}.
\end{equation*}%
Analogously, the (m-j) order principal submatrix matrix with row-column
pairs $i_{j+1},i_{j+2},...,i_{m},i_{1}$\ is symmetrizable, then we have%
\begin{equation*}
\left( \underset{k=j+1}{\overset{m}{\Pi }}a_{i_{k},i_{k+1}}\right)
.a_{i_{1},i_{j+1}}=\left( \underset{k=j+1}{\overset{m}{\Pi }}%
a_{i_{k+1},i_{k}}\right) .a_{i_{j+1},i_{1}},
\end{equation*}%
then, taking in account that $a_{i_{j+1},i_{1}}a_{i_{1},i_{j+1}}\neq 0$,
deduce that%
\begin{equation*}
\left( \underset{k=1}{\overset{j}{\Pi }}a_{i_{k},i_{k+1}}\right) \left( 
\underset{k=j+1}{\overset{m}{\Pi }}a_{i_{k},i_{k+1}}\right) =\left( \underset%
{k=1}{\overset{j}{\Pi }}a_{i_{k+1},i_{k}}\right) \left( \underset{k=j+1}{%
\overset{m}{\Pi }}a_{i_{k+1},i_{k}}\right) 
\end{equation*}%
and this gives (\ref{eq}).
\end{proof}

\begin{remark}
It is shown that a matrix A is symmetrizable if and only if there exists a
diagonal matrix D with positive entries (called Symmetrizer), such that the
matrix D.A is symmetric (see \cite{Fis}).
\end{remark}

Our main result is the following

\begin{theorem}
\label{nonsymmetric I.T.}The eigenvalues of a real symmetrizable matrix A of
order m are all real and interlace with those of any principal submatrix of
order m - 1.
\end{theorem}

The proof of the theorem is based on an equality analogous to (\ref{formula}%
) which is summarized by the following

\begin{lemma}
For any symmetrizable matrix $A$, we have%
\begin{equation}
a_{lk}.\det \left[ \left( a_{i,j}\right) _{\substack{ i\neq l  \\ j\neq k}}%
-\lambda I_{m-1}\right] =a_{kl}.\det \left[ \left( a_{i,j}\right) 
_{\substack{ i\neq k  \\ j\neq l}}-\lambda I_{m-1}\right] ,
\label{formula sign}
\end{equation}%
for all real $\lambda $ and all integers $k,l=1,2,...,m.$
\end{lemma}

\begin{proof}
Simultaneously permuting rows and columns, if necessary, we may assume that $%
k=1,$ $l=2$ and $a_{12}.a_{21}>0$ (i.e.$\neq 0$)$.$ We should prove the
following 
\begin{equation}
P\left( \lambda \right) =Q\left( \lambda \right) ,\text{ for all real }%
\lambda ,  \label{formula sign l=1}
\end{equation}%
where%
\begin{equation*}
P\left( \lambda \right) =a_{12}.\det \left[ \left( a_{i,j}\right) 
_{\substack{ i\neq 1 \\ j\neq 2}}-\lambda I_{m-1}\right] ,
\end{equation*}%
and%
\begin{equation*}
Q\left( \lambda \right) =a_{21}.\det \left[ \left( a_{i,j}\right) 
_{\substack{ i\neq 2 \\ j\neq 1}}-\lambda I_{m-1}\right] .
\end{equation*}%
The left and right parts of formula (\ref{formula sign l=1})\ can be
written, respectively, as follows%
\begin{equation}
P\left( \lambda \right) =a_{12}.\left\vert 
\begin{array}{ccccc}
a_{21} & a_{23} & \cdots  & a_{2,\left( m-1\right) } & a_{2,m} \\ 
a_{31} & a_{33}-\lambda  & \cdots  & a_{3,\left( m-1\right) } & a_{3,m} \\ 
\vdots  & \vdots  & \ddots  & \vdots  & \vdots  \\ 
a_{\left( m-1\right) ,1} & a_{\left( m-1\right) ,3} & \cdots  & a_{\left(
m-1\right) ,\left( m-1\right) }-\lambda  & a_{\left( m-1\right) ,m} \\ 
a_{m,1} & a_{m,3} & \cdots  & a_{m,\left( m-1\right) } & a_{m,m}-\lambda 
\end{array}%
\right\vert ,  \label{P}
\end{equation}%
and%
\begin{equation*}
Q\left( \lambda \right) =a_{21}.\left\vert 
\begin{array}{ccccc}
a_{12} & a_{13} & \cdots  & a_{1,\left( m-1\right) } & a_{1,m} \\ 
a_{32} & a_{33}-\lambda  & \cdots  & a_{3,\left( m-1\right) } & a_{3,m} \\ 
\vdots  & \vdots  & \ddots  & \vdots  & \vdots  \\ 
a_{\left( m-1\right) ,2} & a_{\left( m-1\right) ,3} & \cdots  & a_{\left(
m-1\right) ,\left( m-1\right) }-\lambda  & a_{\left( m-1\right) ,m} \\ 
a_{m,2} & a_{m,3} & \cdots  & a_{m,\left( m-1\right) } & a_{m,m}-\lambda 
\end{array}%
\right\vert .
\end{equation*}%
Since a square matrix and its transpose have the same determinant, then%
\begin{equation}
Q\left( \lambda \right) =a_{21}.\left\vert 
\begin{array}{ccccc}
a_{12} & a_{32} & \cdots  & a_{\left( m-1\right) ,2} & a_{m,2} \\ 
a_{13} & a_{33}-\lambda  & \cdots  & a_{\left( m-1\right) ,3} & a_{m,3} \\ 
\vdots  & \vdots  & \ddots  & \vdots  & \vdots  \\ 
a_{1,\left( m-1\right) } & a_{3,\left( m-1\right) } & \cdots  & a_{\left(
m-1\right) ,\left( m-1\right) }-\lambda  & a_{m,\left( m-1\right) } \\ 
a_{1,m} & a_{3,m} & \cdots  & a_{\left( m-1\right) ,m} & a_{m,m}-\lambda 
\end{array}%
\right\vert .  \label{Q}
\end{equation}%
As $A$ is symmetrizable, then the matrix $A-\lambda I_{m}$ and all its
principal submatrices are symmetrizable for all real $\lambda .$ By
developing the determinants in formulas (\ref{P}) and (\ref{Q}) each one
with respect to its first row for example, we can remark that each
coefficient in the (m-2) degree polynomial $P\left( \lambda \right) $ is
equal to its analogous in the expression of the same degree polynomial $%
Q\left( \lambda \right) $:\newline
Indeed, the coefficient of $\lambda ^{m-2}$ is $\left( -1\right)
^{m-2}a_{12}a_{21}$ for both $P\left( \lambda \right) $ and $Q\left( \lambda
\right) ,$ that of $\lambda ^{m-3}$ is $\left( -1\right) ^{m-3}a_{12}\overset%
{m}{\underset{k=3}{\sum }}\left\vert 
\begin{array}{cc}
a_{21} & a_{2,k} \\ 
a_{k,1} & a_{k,k}%
\end{array}%
\right\vert $ for $P\left( \lambda \right) $ and $\left( -1\right)
^{m-3}a_{21}\overset{m}{\underset{k=3}{\sum }}\left\vert 
\begin{array}{cc}
a_{12} & a_{k,2} \\ 
a_{1,k} & a_{k,k}%
\end{array}%
\right\vert $ for $Q\left( \lambda \right) $, whom are equal since the
subamtrices%
\begin{equation*}
A_{k}=\left( 
\begin{array}{ccc}
a_{11} & a_{12} & a_{1,k} \\ 
a_{21} & a_{22} & a_{2,k} \\ 
a_{k,1} & a_{k,2} & a_{k,k}%
\end{array}%
\right) ,
\end{equation*}%
are symmetrizable for all $k=3,...,m.$ For all integer $l\geq 4$, the
coefficient of $\lambda ^{m-l}$ is 
\begin{equation*}
\left( -1\right) ^{m-l}a_{12}\underset{3\leq k_{1}<k_{2}<...<k_{l-2}\leq m}{%
\sum }\left\vert 
\begin{array}{cccc}
a_{21} & a_{2,k_{1}} & \cdots  & a_{2,k_{l-2}} \\ 
a_{k_{1},1} & a_{k_{1},k_{1}} & \cdots  & a_{k_{1},1} \\ 
\vdots  & \vdots  & \ddots  & \vdots  \\ 
a_{k_{l-2},1} & a_{k_{l-2},k_{1}} & \cdots  & a_{k_{l-2},k_{l-2}}%
\end{array}%
\right\vert 
\end{equation*}%
for $P\left( \lambda \right) $ and 
\begin{equation*}
\left( -1\right) ^{m-l}a_{21}\underset{3\leq k_{1}<k_{2}<...<k_{l-2}\leq m}{%
\sum }\left\vert 
\begin{array}{cccc}
a_{12} & a_{k_{1},2} & \cdots  & a_{k_{l-2},2} \\ 
a_{1,k_{1}} & a_{k_{1},k_{1}} & \cdots  & a_{k_{l-2},k_{1}} \\ 
\vdots  & \vdots  & \ddots  & \vdots  \\ 
a_{1,k_{l-2}} & a_{k_{1},k_{l-2}} & \cdots  & a_{k_{l-2},k_{l-2}}%
\end{array}%
\right\vert 
\end{equation*}%
for $Q\left( \lambda \right) $, whom are equal since the subamtrices%
\begin{equation*}
A_{k_{1},k_{2},...,k_{l-2}}=\left( 
\begin{array}{cccccc}
a_{11} & a_{12} & a_{1,k_{1}} & a_{1,k_{2}} & \cdots  & a_{1,k_{l-2}} \\ 
a_{21} & a_{22} & a_{2,k_{1}} & a_{2,k_{2}} & \cdots  & a_{2,k_{l-2}} \\ 
a_{k_{1},1} & a_{k_{1},2} & a_{k_{1},k_{1}} & a_{k_{1},k_{2}} & \cdots  & 
a_{k_{1},k_{l-2}} \\ 
a_{k_{2},1} & a_{k_{2},2} & a_{k_{2},k_{1}} & a_{k_{2},k_{2}} & \cdots  & 
a_{k_{2},k_{l-2}} \\ 
\vdots  & \vdots  & \vdots  & \vdots  & \ddots  & \vdots  \\ 
a_{k_{l-2},1} & a_{k_{l-2},2} & a_{k_{l-2},k_{1}} & a_{k_{l-2},k_{2}} & 
\cdots  & a_{k_{l-2},k_{l-2}}%
\end{array}%
\right) 
\end{equation*}%
are symmetrizable for all $3\leq k_{1}<k_{2}<...<k_{l-2}\leq m.$ This gives (%
\ref{formula sign l=1}) and ends the proof of the Lemma.
\end{proof}

\begin{proof}
(of Theorem \ref{nonsymmetric I.T.}) Let $\lambda _{1}<\lambda
_{2}<...<\lambda _{m}$ the eigenvalues of the symmetrizable matrix $A$, then
by application of the Lemma and taking in account that $a_{lk}.a_{lk}>0$,
for all $k,l=1,2,...,m,$ we can conclude that the two determinants in (\ref%
{formula sign}) have the same sign for a fixed $p=1,2,...,m$. The Dodgson's
algorithm (\ref{dodgson}) applied to the matrix $A-\lambda _{p}I_{m}$ gives
the formula (\ref{identity}) from which we can deduce, in the same way as
when the matrix $A$ is symmetric, that all principal minors of order $m-1$
of the matrix $A-\lambda _{p}I$ $\ $have the same sign. The rest of the
proof can be obtained by following the same reasoning as in the proof of the
Cauchy interlace theorem for symmetric matrices given is the beginning of
the above section.\newline
To prove, via the Dodgson's algorithm, that the eigenvalues of the
symmetrizable matrix $A$ are all real, we shall do this by induction on $m$:%
\newline
For $m=3,$ the Dodgson's algorithm (\ref{dodgson}) gives%
\begin{equation}
\left. 
\begin{array}{c}
\left\vert 
\begin{array}{ccc}
a_{11}-\lambda  & a_{12} & a_{13} \\ 
a_{21} & a_{22}-\lambda  & a_{23} \\ 
a_{31} & a_{32} & a_{33}-\lambda 
\end{array}%
\right\vert \left( a_{22}-\lambda \right) = \\ 
\\ 
\left\vert 
\begin{array}{cc}
\left\vert 
\begin{array}{cc}
a_{11}-\lambda  & a_{12} \\ 
a_{21} & a_{22}-\lambda 
\end{array}%
\right\vert  & \left\vert 
\begin{array}{cc}
a_{21} & a_{22}-\lambda  \\ 
a_{31} & a_{32}%
\end{array}%
\right\vert  \\ 
\left\vert 
\begin{array}{cc}
a_{12} & a_{13} \\ 
a_{22}-\lambda  & a_{23}%
\end{array}%
\right\vert  & \left\vert 
\begin{array}{cc}
a_{22}-\lambda  & a_{23} \\ 
a_{32} & a_{33}-\lambda 
\end{array}%
\right\vert 
\end{array}%
\right\vert .%
\end{array}%
\right.   \label{dodgson 3}
\end{equation}%
Let $\mu _{1}$ and $\mu _{2}$ the eigenvalues of the submatrix%
\begin{equation*}
\left( 
\begin{array}{cc}
a_{11} & a_{12} \\ 
a_{21} & a_{22}%
\end{array}%
\right) .
\end{equation*}%
Since $a_{12}a_{21}>0$, then $\mu _{1}$ and $\mu _{2}$ are real and we have%
\begin{equation}
\mu _{1}<a_{22}<\mu _{2}.\text{ \ \ }  \label{ineg}
\end{equation}%
By application of the Lemma, we have%
\begin{equation*}
\left\vert 
\begin{array}{cc}
a_{21} & a_{22}-\mu _{i} \\ 
a_{31} & a_{32}%
\end{array}%
\right\vert \left\vert 
\begin{array}{cc}
a_{12} & a_{13} \\ 
a_{22}-\mu _{i} & a_{23}%
\end{array}%
\right\vert >0,\ i=1,\ 2,
\end{equation*}%
which gives%
\begin{equation*}
\left( a_{22}-\mu _{i}\right) P_{3}\left( \mu _{i}\right) <0,\ i=1,\ 2,
\end{equation*}%
where $P_{3}\left( \lambda \right) $ denotes the characteristic polynomial
of $A$ for $m=3$. Using (\ref{ineg}), we deduce that $P_{3}\left( \mu
_{1}\right) <0$ and $P_{3}\left( \mu _{2}\right) >0.$ Since $\underset{%
\lambda \rightarrow -\infty }{\lim }P_{3}\left( \lambda \right) =+\infty $
and $\underset{\lambda \rightarrow +\infty }{\lim }P_{3}\left( \lambda
\right) =-\infty ,$ we deduce that $P_{3}\left( \lambda \right) $ has three
real roots whom are the eigenvalues of the matrix $A$ when $m=3.$\newline
Suppose that the eigenvalues of any symmetrizable matrix until the order $m-1
$ are real and prove that the property is also true for the order $m$. By
application of The Dodgson's algorithm (\ref{dodgson}) to the matrix $%
A-\lambda I_{m}$ and the formula (\ref{formula sign}), we have%
\begin{equation*}
\left. 
\begin{array}{c}
P_{m}\left( \lambda \right) .\det \left[ \left( a_{i,j}-\lambda
I_{m-2}\right) _{\substack{ i\neq k,l \\ j\neq k,l}}\right] \leq  \\ 
\\ 
\det \left[ \left( a_{i,j}-\lambda I_{m-1}\right) _{\substack{ i\neq l \\ %
j\neq l}}\right] \det \left[ \left( a_{i,j}-\lambda I_{m-1}\right) 
_{\substack{ i\neq k \\ j\neq k}}\right] 
\end{array}%
\right. ,
\end{equation*}%
for all real $\lambda $, where $P_{m}\left( \lambda \right) $ denotes the
characteristic polynomial of A. Let $\mu _{1}<\mu _{2}<....<\mu _{m-1}$ the
eigenvalues of $\left[ \left( a_{i,j}-\lambda I_{m-1}\right) _{\substack{ %
i\neq l \\ j\neq l}}\right] $ which is a symmetrizable submmatrix of $%
A-\lambda I_{m},$ then%
\begin{equation*}
P_{m}\left( \mu _{p}\right) \det \left[ \left( a_{i,j}-\mu
_{p}I_{m-2}\right) _{\substack{ i\neq k,l \\ j\neq k,l}}\right] <0,\ \
p=1,...,m-1.
\end{equation*}%
The above inequality is strict because the eigenvalues of $\left(
a_{i,j}-\lambda I_{m-1}\right) _{\substack{ i\neq l \\ j\neq l}}$ interlace
at the same time with those of the two matrices $A$ and $\left[ \left(
a_{i,j}-\lambda I_{m-2}\right) _{\substack{ i\neq k,l \\ j\neq k,l}}\right] $%
.\ Since this last matrix is a submatrix of the $(m-1)$ order matrix $\left[
\left( a_{i,j}-\lambda I_{m-1}\right) _{\substack{ i\neq l \\ j\neq l}}%
\right] $ and then their corresponding eigenvalues interlace, we conclude
that the sequence $\left\{ P_{m}\left( \mu _{p}\right) \right\} _{1\leq
p\leq m-1}$ changes sign $(m-1)$ times. But to deduce that the
characteristic polynomial$\ $of $A$ has m roots, it must change sign $(m+1)$
times. To find the two remaining times we treat two cases:\newline
When $m$ is odd, then the first and last terms of the sequence $\left\{
P_{m}\left( \mu _{p}\right) \right\} _{1\leq p\leq m-1}$ are respectively%
\begin{equation*}
P_{m}\left( \mu _{1}\right) <0\text{ and }P_{m}\left( \mu _{m-1}\right) >0.
\end{equation*}%
Then using the limits 
\begin{equation*}
\underset{\lambda \rightarrow -\infty }{\lim }P_{m}\left( \lambda \right)
=+\infty \text{ and }\underset{\lambda \rightarrow +\infty }{\lim }%
P_{m}\left( \lambda \right) =-\infty ,
\end{equation*}%
we deduce that the characteristic polynomial$\ $of $A$ has $m$ roots.\newline
When $m$ is even, they are respectively%
\begin{equation*}
P_{m}\left( \mu _{1}\right) <0\text{ and }P_{m}\left( \mu _{m-1}\right) <0,
\end{equation*}%
then using the limits 
\begin{equation*}
\underset{\lambda \rightarrow \pm \infty }{\lim }P_{m}\left( \lambda \right)
=+\infty ,
\end{equation*}%
we get the two remaining roots of $P_{m}\left( \lambda \right) $. This ends
the proof of Theorem (\ref{nonsymmetric I.T.}).
\end{proof}

\begin{remark}
When the symmetrizable matrix $A$\ has a multiple eigenvalue $\lambda $\ of
algebraic multiplicity equal to $r$, then the characteristic polynomial (\ref%
{Pk})\ of the matrix $\left[ \left( a_{i,j}-\lambda I_{m-1}\right) 
_{\substack{ i\neq k \\ j\neq k}}\right] $\ satisfies (\ref{Alter}) for
p=1,..., (m-r+1) and Theorem \ref{nonsymmetric I.T.} remains valid.
\end{remark}

\begin{acknowledgement}
The author gratefully acknowledge Qassim University, represented by the
Deanship of Scientific Research, on the material support for this research
under the number (3388 ) during the academic year 1436 AH / 2015 AD.
\end{acknowledgement}


\begin{thebibliography}{99}
\bibitem{Bel} R. Bellman, Introduction to Matrix Analysis, 2nd ed.,
McGraw-Hill Book Co., New York, 1970.

\bibitem{Car} D. H. Carlson, On Real Eigenvalues of Complex Matrices,
Pacific Journal of Mathematics, Vol. 15, No. 4, 1965.

\bibitem{Cau} A. Cauchy, Cours d'Analyse de l'Ecole Polytechnique,In oeuvres
completes, Volumes 2 et 3 (1821).

\bibitem{Dod} C.L.Dodgson, Condensation of Determinants, Proc. London Math.
Soc.15(1866),150-155.

\bibitem{Fis} S. Fisk, A very short proof of Cauchy's interlace theorem for
eigenvalues of Hermitian matrices, arXiv:math/0502408v1 [math.CA]

\bibitem{Fom} S. Fomin, A. Zelevinsky, Cluster Algebras II: Finite Type
Classification, Invent. Math., (154) 63--121 (2003).

\bibitem{Gol-Van} G. H. Golub and C. F. Van Loan, Matrix Computations, 2nd
ed., Johns Hopkins University Press, Baltimore, 1989.

\bibitem{Hor-Joh} R. A. Horn and C. R. Johnson, Matrix Analysis, Cambridge
University Press, New York, 1985.

\bibitem{Ike-Ina-Miy} Y. Ikebe, T. Inagaki, and S. Miyamoto, The
monotonicity theorem, Cauchy's interlace theorem and the Courant-Fischer
theorem, Amer. Math. Monthly 94 (1987), no. 4, 352-- 354. MR 88a:15035. Zbl
623.15010.

\bibitem{Kou ELA} S. Kouachi, Eigenvalues and Eigenvectors of Tridiagonal
matrices, Electronic Journal of linear Algebra, Vol 15 (April 2006) pp.
115-133.

\bibitem{Kou IJPAM} S. Kouachi, Explicit Eigenvalues of Several Perturbed
Pentadiagonal Matrices, Accepted for publication in International Journal of
Pure and Applied Mathematics.

\bibitem{Kou LJM} S. Kouachi, Explicit Eigenvalues of some perturbed
Heptadiagonal Matrices via recurrent sequences, Lobachevskii Journal of
Mathematics, Vol. 36, issue 1, pp 28-37(2015).

\bibitem{Par} B. N. Parlett, The Symmetric Eigenvalue Problems,
Prentice-Hall, Englewood Cliffs, NJ, 1980.
\end{thebibliography}
\end{document}